\documentclass{amsart}    
\usepackage{amsthm, amsmath, amscd, amssymb}
\theoremstyle{plain}
\newtheorem{theorem}{Theorem}[section]
\newtheorem{lemma}[theorem]{Lemma}
\newtheorem{proposition}[theorem]{Proposition}
\newtheorem{corollary}[theorem]{Corollary}
\newtheorem{conjecture}[theorem]{Conjecture}
\theoremstyle{remark}

\newtheorem{remark}[theorem]{Remark}
\newtheorem{example}[theorem]{Example}

\usepackage[mathscr]{eucal}
\usepackage{graphics, graphpap}
\usepackage{array, tabularx, longtable}

\title{\bf On a conjecture of Kontsevich and Soibelman}  
\author{L\^e Quy Thuong}
\dedicatory{\it Dedicated to Professor H\`a Huy Vui on the occasion of his sixtieth birthday}
\address{Ecole Normale Sup\'erieure, D\'epartment de Math\'ematiques et Applications UMR 8553 CNRS, 45 rue d'Ulm, 75230 Paris cedex 05, France }
\email{Thuong.Le@ens.fr}
\address{({\it Current}) Institut de Math\'ematique de Jussieu, UMR 7586 CNRS, 4 place Jussieu, 75005 Paris, France }
\email{leqthuong@math.jussieu.fr}
\keywords{arc spaces, motivic Milnor fiber, motivic zeta function, Newton polyhedron}
\subjclass[2000]{Primary 14B05, 14B07, 14J17, 32S05, 32S30, 32S55}

\begin{document}           
\begin{abstract}
We consider a conjecture of Kontsevich and Soibelman which is regarded as a foundation of their theory of motivic Donaldson-Thomas invariants for non-commutative $3d$ Calabi-Yau varieties. We will show that, in some certain cases, the answer to this conjecture is positive.
\end{abstract}
\maketitle                 

\section{Introduction}\label{sec1}
In \cite{KS}, Kontsevich and Soibelman introduce and give discussions on the motivic Donaldson-Thomas invariants which are defined for non-commutative $3d$ Calabi-Yau varieties and take values in certain Grothendieck groups of algebraic varieties. One of the main objectives of \cite{KS} is to define the motivic Hall algebra which generates To$\ddot{\rm e}$n's notion of the derived Hall algebra (cf. \cite{To}). For $\mathcal{C}$ an ind-constructible triangulated $A_{\infty}$-category over a field $\kappa$, the motivic Hall algebra $H(\mathcal{C})$ is constructed to become a graded associative algebra, which admits for each strict sector $V$ an element $A_V^{\text{Hall}}$ invertible in the completed motivic Hall algebra and satisfying the Factorization Property. It is believed that, in the case of $3d$ Calabi-Yau category, there is a homomorphism $\Phi$ of the motivic Hall algebra into the motivic quantum torus defined in terms of the motivic Milnor fiber of the potential. Then the motivic Donaldson-Thomas invariants appear as the collection of the images of $A_V^{\text{Hall}}$ under the homomorphism $\Phi$. \\
\indent In fact, a central role in the existence of $\Phi$ is played by the following conjecture. Assume that the characteristic of $\kappa$ is zero. Let $F$ be a formal series on the affine space $\mathbb{A}_{\kappa}^d=\mathbb{A}_{\kappa}^{d_1}\times_{\kappa}\mathbb{A}_{\kappa}^{d_2}\times_{\kappa}\mathbb{A}_{\kappa}^{d_3}$, depending on a constructible way on finitely many extra parameters, such that $F(0,0,0)=0$ and $F$ has degree zero with respect to the diagonal action of the multiplicative group $\mathbb{G}_{m,\kappa}$ with the weights $(1,-1,0)$. In particular, $F(x,0,0)$ is the zero function on $\mathbb{A}_{\kappa}^{d_1}$. We denote by $X_0(F)$ the set of the zeros of $F$ on $\mathbb{A}_{\kappa}^d$. Consider the natural inclusions $i_1:\mathbb{A}_{\kappa}^{d_1}\times_{\kappa} \mathbb{G}_{m,\kappa}\rightarrow X_0(F)\times_{\kappa} \mathbb{G}_{m,\kappa}$ and $i_0: \{0\}\times_{\kappa} \mathbb{G}_{m,\kappa}\rightarrow X_0(F)\times_{\kappa} \mathbb{G}_{m,\kappa}$. Consider the motivic Milnor fiber $\mathcal{S}_F$ of $F$ in the ring $\mathscr{M}_{X_0(F)\times_{\kappa} \mathbb{G}_{m,\kappa}}^{\mathbb{G}_{m,\kappa}}$, the localisation of the relative Grothendieck ring defined in \cite{GLM1} and \cite{GLM2}. Denote by $h$ the function on $\mathbb{A}_{\kappa}^{d_3}$ defined by $h(z)=F(0,0,z)$. We write $\mathcal{S}_{h,0}$ for the pullback $i_0^*\mathcal{S}_h$. We denote by integral $\int_{\mathbb{A}_{\kappa}^{d_1}}$ the pushforward of the canonical morphism $\pi: \mathbb{A}_{\kappa}^{d_1}\times_{\kappa} \mathbb{G}_{m,\kappa}\rightarrow Spec(k)\times_{\kappa} \mathbb{G}_{m,\kappa}$. 
\begin{conjecture}[\cite{KS}]\label{conj}
With the previous notations and hypotheses, the following formula holds in $\mathscr{M}_{\mathbb{G}_{m,\kappa}}^{\mathbb{G}_{m,\kappa}}$:
$$\int_{\mathbb{A}_{\kappa}^{d_1}}i_1^*\mathcal{S}_F=\mathbb{L}^{d_1}\mathcal{S}_{h,0}.$$
\end{conjecture}
\indent In this paper, we consider the conjecture in some special (actually quite general) cases, namely, when $F$ is a composition of a polynomial in two variables and a pair of two regular functions (Theorem \ref{thm5.1}), or $F$ has the form $F(x,y,z)=g(x,y,z)+h(z)^{\ell}$ with $\ell$ sufficiently large (function of Steenbrink type, Theorem \ref{thm5.6}) under some additional conditions of nondegeneracy with respect to Newton polyhedron (this would be the general case for the conjecture if we did not assume $\ell$ sufficiently large). For these cases, we use previous results of Guibert, Loeser and Merle (\cite{GLM1}, \cite{GLM3}) for the motivic Milnor fiber of composite functions or functions of Steenbrink type. We also use in an important way, via Proposition \ref{prop4.5}, the explicit computation of the motivic Milnor fiber of a regular function via its Newton polyhedron (suggested by \cite{G}). These lead to the positive answer to the conjecture in the cases considered. \\
\indent This work was suggested by Fran\c cois Loeser, my advisor, who proposed me to consider the conjecture firstly in the case of composition $f(g_1,g_2)$ and encouraged me in each step of proof. I am deeply grateful to him for these, for his suggestions of method approaching to the solution and for his help in preparing the manuscript. I would like to thank the referee for his contributions to the paper which make it become more readable.

\section{Motivic zeta function and Motivic Milnor fiber}\label{sec2}
Let us recall some basic notations in the theory of motivic integration which will be used in this paper. For references, we follow \cite{DL1}, \cite{DL2}, \cite{DL3}, \cite{G}, \cite{GLM1} and \cite{GLM2}.
\subsection{}
Let $\kappa$ be a field of characteristic zero. For a variety $X$ over $\kappa$, we denote by $\mathscr{L}_m(X)$ the space of $m$-arcs on $X$, and by $\mathscr{L}(X)$ a limit of the projective system of spaces $\mathscr{L}_m(X)$ and (canonical) morphisms $\mathscr{L}_l(X)\rightarrow \mathscr{L}_m(X)$ ($l\geq m$). In this paper, we use the notation $\pi_m$ for the canonical morphism $\mathscr{L}(X)\to\mathscr{L}_m(X)$. The $\mathbb{G}_{m,\kappa}$-action on $\mathscr{L}_m(X)$ and $\mathscr{L}(X)$ is given by $a\cdot\varphi(t)=\varphi(at)$. For the notation $\mathscr{M}_X$, we can find in \cite{GLM1}. As in \cite{GLM2}, we denote by $\mathscr{M}_{X\times_{\kappa} \mathbb{G}_{m,\kappa}}^{\mathbb{G}_{m,\kappa}}$ the localisation at $\mathbb{L}$ of the relative Grothendieck ring of $\mathbb{G}_{m,\kappa}$-equivariant morphisms $Y\rightarrow X\times_{\kappa} \mathbb{G}_{m,\kappa}$ endowed with a monomial $\mathbb{G}_{m,\kappa}$-action, where $\mathbb{L}$ is the class of the line bundle $\mathbb{A}_{X\times_{\kappa} \mathbb{G}_{m,\kappa}}^1$.\\
\indent From now on, the group scheme $\mathbb{G}_{m,\kappa}={\rm Spec} (\kappa[t,t^{-1}])$ will be written simply as $\mathbb{G}$.

\subsection{Motivic zeta function and Motivic Milnor fiber}
Let $X$ be a smooth variety over $\kappa$ of pure dimension $n$, and let $g:X\rightarrow \mathbb{A}_{\kappa}^1$ be a function on $X$ and $X_0(g)$ the zero locus of $g$. For $m\geq 1$, we define
$$\mathscr{X}_m(g):=\{\varphi\in\mathscr{L}_m(X)\ |\ ord_tg(\varphi)=m\}.$$
Note that this variety is invariant by the $\mathbb{G}$-action on $\mathscr{L}_m(X)$. Furthermore, $g$ induces a morphism $g_m:\mathscr{X}_m(g)\rightarrow \mathbb{G}$, assigning to a point $\varphi$ in $\mathscr{L}_m(X)$ the coefficient $ac(g(\varphi))$ of $t^m$ in $g(\varphi(t))$, which we also denote by $ac(g)(\varphi)$. This morphism is a diagonally monomial of weight $m$ with respect to the $\mathbb{G}$-action on $\mathscr{X}_m(g)$ since $g(s\cdot\varphi)=s^mg_m(\varphi)$. We thus consider the class $[\mathscr{X}_m(g)]$ of $\mathscr{X}_m(g)$ in $\mathscr{M}_{X_0(g)\times_{\kappa} \mathbb{G}}^{\mathbb{G}}$. We can now consider the {\it motivic zeta function}
$$Z_g(T):=\sum_{m\geq 1}[\mathscr{X}_m(g)]\mathbb{L}^{-mn}T^m$$
in $\mathscr{M}_{X_0(g)\times_{\kappa} \mathbb{G}}^{\mathbb{G}}[[T]]$. Note that $Z_g=0$ if $g=0$ on $X$.\\
\indent By using a log-resolution of $X_0(g)$, Denef and Loeser proved in \cite{DL1} and \cite{DL3} that $Z_g(T)$ is a rational series in $\mathscr{M}_{X_0(g)\times_{\kappa} \mathbb{G}}^{\mathbb{G}}[[T]]_{sr}$ (cf. (\ref{2.3})) and they also showed that one can consider the limit $\lim_{T\rightarrow\infty}Z_g(T)$ in $\mathscr{M}_{X_0(g)\times_{\kappa} \mathbb{G}}^{\mathbb{G}}$. Then the {\it motivic Milnor fiber} of $g$ is defined as
$$\mathcal{S}_g:=-\lim_{T\rightarrow\infty}Z_g(T).$$

\subsection{Rational series and their limits}\label{2.3}
Let $A$ be one of the rings $\mathbb{Z}[\mathbb{L},\mathbb{L}^{-1}]$, $\mathbb{Z}[\mathbb{L},\mathbb{L}^{-1},(1/(1-\mathbb{L}^{-i}))_{i>0}]$, $\mathscr{M}_{S\times_{\kappa} \mathbb{G}}^{\mathbb{G}}$. We denote by $A[[T]]_{sr}$ the $A$-submodule of $A[[T]]$ generated by $1$ and by finite products of terms $p_{e,i}(T)=\mathbb{L}^eT^i/(1-\mathbb{L}^eT^i)$ with $e$ in $\mathbb{Z}$ and $i$ in $\mathbb{N}_{>0}$. There is a unique $A$-linear morphism 
$$\lim_{T\rightarrow\infty}: A[[T]]_{sr}\rightarrow A$$
such that 
$$\lim_{T\rightarrow\infty}\Big(\prod_{i\in I}p_{e_i,j_i}(T)\Big)=(-1)^{|I|}$$
for every family $((e_i,j_i))_{i\in I}$ in $\mathbb{Z}\times\mathbb{N}_{>0}$ with $I$ finite (possibly empty).\\
\indent From now on, we will use the following notations
$$\mathbb{R}_{\geq 0}^I:=\{a=(a_1,\dots,a_n)\in\mathbb{R}_{\geq 0}^n\ |\ a_i=0\ \text{for}\ i\not\in I\},$$
and
$$\mathbb{R}_{>0}^I:=\{a=(a_1,\dots,a_n)\in\mathbb{R}_{\geq 0}^n\ |\ a_i=0\ \text{iff}\ i\not\in I\},$$
for $I$ a subset of $\{1,\dots,n\}$. The sets $\mathbb{Z}_{\geq 0}^I$, $\mathbb{Z}_{> 0}^I$ and $\mathbb{N}_{> 0}^I$ are defined similarly.\\
\indent Let $\Delta$ be a rational polyhedral convex cone in $\mathbb{R}_{>0}^I$ and let $\overline{\Delta}$ denote its closure in $\mathbb{R}_{\geq 0}^I$ with $I$ a finite set. Let $l$ and $l'$ be two integer linear forms on $\mathbb{Z}^I$ positive on $\overline{\Delta}\setminus\{0\}$. Let us consider the series
$$S_{\Delta,l,l'}(T):=\sum_{k\in\Delta\cap\mathbb{N}_{>0}^I}\mathbb{L}^{-l'(k)}T^{l(k)}$$
in $\mathbb{Z}[\mathbb{L},\mathbb{L}^{-1}][[T]]$. In this paper, we will use the following lemma.
\begin{lemma}[\cite{G}]\label{lem2.1}
With previous notations and hypotheses, assuming that $\Delta$ is open in its linear span $\overline{\Delta}$, the series $S_{\Delta,l,l'}(T)$ lies in $\mathbb{Z}[\mathbb{L},\mathbb{L}^{-1}][[T]]_{sr}$ and
$$\lim_{T\rightarrow\infty}S_{\Delta,l,l'}(T)=(-1)^{\dim(\Delta)}.$$
\end{lemma}


\section{Newton polyhedron of a regular function}\label{sec3}
\subsection{Newton polyhedron}
Let $g(x)=\sum_{\alpha\in\mathbb{N}^n}a_{\alpha}x^{\alpha}$ be a polynomial in $n$ variables $x=(x_1,\dots,x_n)$ such that $g(0)=0$. We denote by $supp(g)$ the set of exponents $\alpha$ in $\mathbb{N}^n$ with $a_{\alpha}\not=0$. The Newton polyhedron $\Gamma$ of $g$ is the convex hull of $supp(g)+\mathbb{R}_{\geq 0}^n$. For a compact face $\gamma$ of $\Gamma$, we denote by $g_{\gamma}$ the following quasi-homogenous polynomial
$$g_{\gamma}(x)=\sum_{\alpha\in\gamma}a_{\alpha}x^{\alpha}.$$
We say $g$ is {\it non-degenerate} with respect to its Newton polyhedron $\Gamma$ if, for every compact face $\gamma$ of $\Gamma$, the {\it face} function $g_{\gamma}$ is smooth on $\mathbb{G}^n$.\\
\indent To the Newton polyhedron $\Gamma$ we associate a function $l_{\Gamma}$ which assigns to a vector $a$ in $\mathbb{R}_{\geq 0}^n$ the value $\inf_{b\in\Gamma}\langle a,b\rangle$, with $\langle a,b\rangle$ being the standard inner product of $a$ and $b$. For $a$ in $\mathbb{R}_{\geq 0}^n$, we denote by $\gamma_a$ the face of $\Gamma$ on which the restriction of the function $\langle a,.\rangle$ on $\Gamma$ attains its minimum, i.e., $b\in\Gamma$ is in $\gamma_a$ if and only if 
$$\langle a,b\rangle=l_{\Gamma}(a)=\min_{b\in\Gamma}\langle a,b\rangle.$$
\indent For $a=0$ in $\mathbb{R}_{\geq 0}^n$, $\gamma_a=\Gamma$. If $a\not=0$, $\gamma_a$ is a proper face of $\Gamma$. Furthermore, $\gamma_a$ is a compact face of $\Gamma$ if and only if $a$ is in $\mathbb{R}_{>0}^n$. For any face $\gamma$ of the Newton polyhedron $\Gamma$, we denote by $\sigma(\gamma)$ the cone $\{a\in\mathbb{R}_{\geq 0}^n\ |\ \gamma_a=\gamma\}$. Then its closure is given by $\overline{\sigma}(\gamma)=\{a\in\mathbb{R}_{\geq 0}^n\ |\ \gamma_a\supset\gamma\}$.\\
\indent A {\it fan} $\mathscr{F}$ is a finite set of rational polyhedral cones such that every face of a cone of $\mathscr{F}$ is also a cone of $\mathscr{F}$, and the intersection of two arbitrary cones of $\mathscr{F}$ is the common face of them. It is easily shown that, when $\gamma$ runs over the faces of $\Gamma$, $\overline{\sigma}(\gamma)$ form a fan in $\mathbb{R}_{\geq 0}^n$ partitioning $\mathbb{R}_{\geq 0}^n$ into rational polyhedral cones.

\subsection{Partition of $\mathbb{R}_{\geq 0}^{n_1}\times\mathbb{R}_{>0}^{n_2}$ with respect to $g$}
 Write $n=n_1+n_2$ with $n_1\geq 0$, $n_2\geq 0$. Let $g$ be a function on $\mathbb{A}_{\kappa}^n$ which is non-degenerate with respect to the Newton polyhedron $\Gamma$ of $g$. Let $\gamma$ be a compact face of $\Gamma$. A proper face $\epsilon$ of $\Gamma$ is said to {\it lean} on $\gamma$ if there exists a subset $I$ of $\{1,\dots,n\}$ such that
$$\epsilon=\gamma+\mathbb{R}_{\geq 0}^I=\{a+b\ |\ a\in\gamma, b\in\mathbb{R}_{\geq 0}^I\}.$$
Note that $\dim(\epsilon)=\dim(\gamma)+|I|$. Clearly, the face $\epsilon$ is non-compact when $I$ is nonempty. The following lemmas are trivial.
\begin{lemma}\label{lem3.4}
If $\gamma+\mathbb{R}_{\geq 0}^I$ is a face leaning on a compact face $\gamma$ of $\Gamma$, then for every $J$ subset of $I$, $\gamma+\mathbb{R}_{\geq 0}^J$ is also a face of $\Gamma$ leaning on $\gamma$.
\end{lemma}
\indent Notice that if $I=\emptyset$ the face $\gamma+\mathbb{R}_{\geq 0}^I$ reduces to the compact face $\gamma$. If $\epsilon=\gamma+\mathbb{R}_{\geq 0}^I$, we denote $\sigma_{\gamma,I}:=\sigma(\epsilon)$. It is clear that $\dim(\sigma_{\gamma,I})=n-|I|-\dim(\gamma)$.
\begin{lemma}\label{lem3.5}
If $\sigma_{\gamma,I}$ is contained in $\mathbb{R}_{\geq 0}^{n_1}\times\mathbb{R}_{>0}^{n_2}$, then for every $J$ subset of $I$, $\sigma_{\gamma,J}$ is contained in $\mathbb{R}_{\geq 0}^{n_1}\times\mathbb{R}_{>0}^{n_2}$. Moreover, $\sigma_{\gamma,I}$ is a face of $\sigma_{\gamma,J}$.
\end{lemma}
\begin{lemma}\label{lem3.8}
Assume that $\gamma$ is a compact face and $\epsilon=\gamma+\mathbb{R}^I$ is a face of $\Gamma$. Then $\sigma_{\gamma,I}$ is contained in $\mathbb{R}_{\geq 0}^{n_1}\times\mathbb{R}_{>0}^{n_2}$ if and only if $I$ is a subset of $\{1,\dots,n_1\}$.
\end{lemma}
\indent Fix a compact face $\gamma$ of $\Gamma$. Let $M$ be a maximal element (in the inclusion relation) of the family of the subsets of $\{1,\dots,n_1\}$ such that $\gamma+\mathbb{R}_{\geq 0}^M$ is a face of $\Gamma$ (thus, by Lemma \ref{lem3.8}, $\sigma_{\gamma,M}$ is contained in $\mathbb{R}_{\geq 0}^{n_1}\times\mathbb{R}_{>0}^{n_2}$).  Then, for every $I$ subset of $M$, $\gamma+\mathbb{R}_{\geq 0}^I$ is a face of $\Gamma$ due to Lemma \ref{lem3.4}, and $\sigma_{\gamma,I}$ is contained in $\mathbb{R}_{\geq 0}^{n_1}\times\mathbb{R}_{>0}^{n_2}$ by Lemma \ref{lem3.5}. We thus have proved the following result.
\begin{proposition}\label{prop3.6}
There exists a canonical fan in $\mathbb{R}_{\geq 0}^{n_1}\times\mathbb{R}_{>0}^{n_2}$ with respect to $g$ partitioning it into the cones $\sigma_{\gamma,I}$, where $I$ runs over the subsets of $M$, $M$ runs over the maximal subsets of $\{1,\dots,n_1\}$ such that $\gamma+\mathbb{R}_{\geq 0}^M$ is a face of $\Gamma$, and $\gamma$ runs over the compact faces of $\Gamma$.
\end{proposition}
\begin{example}\label{ex3.8}
Consider a function $g(x_1,\dots,x_n)$ with $\Gamma_g$ having a unique vertex $P$. Then the $k$-dimensional faces of $\Gamma$ leaning on $P$ have the form
$$P+\mathbb{R}_{\geq 0}^I$$
with $I$ subsets of $\{1,\dots,n\}$ and $|I|=k$, for $k=0,\dots,n-1$. We deduce from Lemma \ref{lem3.8} that the canonical partition of $\mathbb{R}_{\geq 0}^{n_1}\times\mathbb{R}_{>0}^{n_2}$ with respect to $g$ is given by the cones $\sigma_{P,I}$, with $I$ subsets of $\{1,\dots,n_1\}$.
\end{example}
\begin{remark}\label{rem3.10}
In the case $n_1=0$, we reduce to the work by Guibert (cf. \cite{G}). More clearly, for each compact face $\gamma$ of $\Gamma$, all the maximal subsets $M$ of $\{1,\dots,n\}$, of which $\gamma+\mathbb{R}_{\geq 0}^M$ is a face of $\Gamma$ and $\sigma_{\gamma,M}\subset \mathbb{R}_{>0}^{n}$, are empty.
\end{remark}

\section{Computation of $i_1^*\mathcal{S}_g$ and $\int_{\mathbb{A}_{\kappa}^{d_1}}i_1^*\mathcal{S}_g$}
Consider a regular function $g$ on $\mathbb{A}_{\kappa}^n$. We assume that $g$ is non-degenerate with respect to its Newton polyhedron $\Gamma$. Denote by $i_1$ the natural inclusion $\mathbb{A}_{\kappa}^{n_1}\hookrightarrow \mathbb{A}_{\kappa}^n$ or $\mathbb{A}_{\kappa}^{n_1}\times_{\kappa} \mathbb{G}\hookrightarrow \mathbb{A}_{\kappa}^n\times_{\kappa} \mathbb{G}$.

\subsection{The motivic zeta function $Z_{g}(T)$}
We identify the arc space $\mathscr{L}(\mathbb{A}_{\kappa}^n)$ with the space of formal power series $k[[t]]^n$ via the system of coordinates $x_1,\dots,x_n$. For every arc $\varphi\in \mathscr{L}(\mathbb{A}_{\kappa}^n)$ we note $ord_tx(\varphi)=(ord_tx_1(\varphi),\dots,ord_tx_n(\varphi))$. For every $m\in\mathbb{N}_{>0}$ and $a\in\mathbb{N}^n$ we set
$$\mathscr{X}_{a,m}(g)=\mathscr{X}_m(g)\cap \pi_m(\mathscr{X}_a),$$
where the spaces $\mathscr{X}_m(g)$ and $\mathscr{X}_a$ are defined as follows
\begin{align*}
\mathscr{X}_m(g)&=\{\varphi\in \mathscr{L}_m(\mathbb{A}_{\kappa}^n)\ |\ ord_tg(\varphi)=m\},\\
\mathscr{X}_a&=\{\varphi\in \mathscr{L}(\mathbb{A}_{\kappa}^n)\ |\ ord_tx(\varphi)=a\}.
\end{align*}
It is clear that $\mathscr{X}_{a,m}(g)$ is a variety over $X_0(g)\times_{\kappa} \mathbb{G}$ in which the morphism to $X_0(g)$ is induced by the canonical morphism $\mathscr{L}_m(\mathbb{A}_{\kappa}^n)\rightarrow \mathbb{A}_{\kappa}^n$ and the morphism to $\mathbb{G}$ is the morphism $ac(g)$. Note that $\mathscr{X}_{a,m}(g)$ is invariant by the $\mathbb{G}$-action on $\mathscr{L}_m(\mathbb{A}_{\kappa}^n)$.\\
\indent For every $a\in\mathbb{N}^n$ and $\varphi\in\mathscr{X}_a$, $ord_tg(\varphi)\geq l_{\Gamma}(a)$ by the definition of $l_{\Gamma}$. Furthermore, $\mathscr{X}_m(g)$ can be expressed as a disjoint union $\bigcup_{a\in\mathbb{N}^n}\mathscr{X}_{a,m}(g)$ of the subspaces $\mathscr{X}_{a,m}(g)$ for $a$ in $\mathbb{N}^n$. Then the motivic zeta function $Z_{g}(T)$ of $g$ can be written in the following form
\begin{align*}
Z_{g}(T)
=&\sum_{a\in\mathbb{N}^n}\sum_{m\geq l_{\Gamma}(a)}[\mathscr{X}_{a,m}(g)]\mathbb{L}^{-nm}T^m\\
=&\sum_{a\in\mathbb{N}^n}\Big([\mathscr{X}_{a,l_{\Gamma}(a)}(g)]\mathbb{L}^{-nl_{\Gamma}(a)}T^{l_{\Gamma}(a)}+\sum_{m\geq l_{\Gamma}(a)+1}[\mathscr{X}_{a,m}(g)]\mathbb{L}^{-nm}T^m\Big)\\
=&:Z^0(T)+Z^1(T).
\end{align*}
\indent There is a canonical partition of $\mathbb{R}_{\geq 0}^n$ into the rational polyhedral cones $\sigma(\gamma)$ with $\gamma$ running over the proper faces of $\Gamma$, so we deduce that
\begin{align*}
Z^0(T)&=\sum_{\gamma}\sum_{a\in\sigma(\gamma)}[\mathscr{X}_{a,l_{\Gamma}(a)}(g)]\mathbb{L}^{-nl_{\Gamma}(a)}T^{l_{\Gamma}(a)},\\
Z^1(T)&=\sum_{\gamma}\sum_{a\in\sigma(\gamma)}\sum_{k\geq 1}[\mathscr{X}_{a,l_{\Gamma}(a)+k}(g)]\mathbb{L}^{-n(l_{\Gamma}(a)+k)}T^{l_{\Gamma}(a)+k},
\end{align*}
where the sum $\sum_{\gamma}$ runs over the proper faces $\gamma$ of $\Gamma$.

\subsection{Computation of $i_1^*Z_g(T)$}
Assume that $g$ satisfies the additional condition that $\mathbb{A}_{\kappa}^{n_1}$ is naturally included in $X_0(g)$ via the morphism $i_1$. To compute $i_1^*Z_g(T)$, we consider the canonical fan in $\mathbb{R}_{\geq 0}^{n_1}\times\mathbb{R}_{>0}^{n_2}$ with respect to $g$. Denote by $\Gamma_c$ the set of compact faces of $\Gamma$, by $\mathfrak{M}_{\gamma}$ the set of maximal subsets $M$ of $\{1,\dots,n_1\}$ such that $\gamma+\mathbb{R}_{\geq 0}^M$ is a face of $\Gamma$. By Proposition \ref{prop3.6}, we can partition $\mathbb{R}_{\geq 0}^{n_1}\times\mathbb{R}_{>0}^{n_2}$ into the cones $\sigma_{\gamma,I}$, with $I$ subsets of $M$, $M$ in $\mathfrak{M}_{\gamma}$ and $\gamma$ in $\Gamma_c$. Assume that $\mathfrak{M}_{\gamma}=\{M_1,\dots,M_p\}$. We denote by $\mathfrak{S}_{\gamma}$ the family of subsets of one of the sets $M_1,\dots,M_p$.  Then we have
\begin{align*}
i_1^*Z^0(T)&=i_1^*\Big(\sum_{\gamma\in\Gamma_c}\sum_{I \in\mathfrak{S}_{\gamma}}\sum_{a\in \sigma_{\gamma,I}}[\mathscr{X}_{a,l_{\Gamma}(a)}(g)]\mathbb{L}^{-nl_{\Gamma}(a)}T^{l_{\Gamma}(a)}\Big),\\
i_1^*Z^1(T)&=i_1^*\Big(\sum_{\gamma\in\Gamma_c}\sum_{I\in\mathfrak{S}_{\gamma}}\sum_{a\in \sigma_{\gamma,I}}\sum_{k\geq 1}[\mathscr{X}_{a,l_{\Gamma}(a)+k}(g)]\mathbb{L}^{-n(l_{\Gamma}(a)+k)}T^{l_{\Gamma}(a)+k}\Big).
\end{align*}

\subsection{Class of $\mathscr{X}_{a,m}(g)$}
For a compact face $\gamma$ of $\Gamma$, consider the variety $X_{\gamma}:=\mathbb{G}^n\setminus g_{\gamma}^{-1}(0)$ endowed with a $\mathbb{G}$-action as follows: if $\gamma=\gamma_a$, $a=(a_1,\dots,a_n)$ then we set
$$s\cdot(\xi_1,\dots,\xi_n)=(s^{a_1}\xi_1,\dots,s^{a_n}\xi_n).$$
For each compact $\gamma$ and $I$ in $\mathfrak{S}_{\gamma}$, consider the morphism 
$$g_{\gamma,I}: X_{\gamma}=\mathbb{G}^n\setminus g_{\gamma}^{-1}(0)\rightarrow X_0(g)\times_{\kappa} \mathbb{G}$$
given by 
$$g_{\gamma,I}(\xi_1,\dots,\xi_n)=\big((\hat{\xi}_1,\dots,\hat{\xi}_n),g_{\gamma}(\xi_1,\dots,\xi_n)\big),$$
where $\hat{\xi}_i$ is defined as follows
\begin{equation*}
\hat{\xi}_i=
\begin{cases}
\xi_i \quad \text{if}\ i\in I\\
0\quad \text{otherwise}. 
\end{cases}
\end{equation*}
The first projection $X_{\gamma}\rightarrow X_0(g)$ is $\mathbb{G}$-equivariant in an obvious manner, and for $\gamma=\gamma_a$, the second $X_{\gamma}\rightarrow \mathbb{G}$ is diagonally monomial of weight $l_{\Gamma}(a)$ with respect to the $\mathbb{G}$-action since $g_{\gamma}(s\cdot(\xi_1,\dots,\xi_n))=s^{l_{\Gamma}(a)}g_{\gamma}(\xi_1,\dots,\xi_n)$ for any $s$ in $\mathbb{G}$. This defines a class $[g_{\gamma,I}: X_{\gamma}\rightarrow X_0(g)\times_{\kappa} \mathbb{G}]$ in $\mathscr{M}_{X_0(g)\times_{\kappa} \mathbb{G}}^{\mathbb{G}}$, which we denote by $\Phi_{\gamma,I}$. Notice that $\Phi_{\gamma,I}$ does not depend on the action thanks to the construction of the Grothendieck group (cf. \cite{GLM1}, \cite{GLM2}).\\
\indent We denote by $\Psi_{\gamma,I}$ the class in $\mathscr{M}_{X_0(g)\times_{\kappa} \mathbb{G}}^{\mathbb{G}}$ of the morphism
$$g_{\gamma}^{-1}(0)\times_{\kappa} \mathbb{G}\rightarrow X_0(g)\times_{\kappa} \mathbb{G},$$
which maps $\big((\xi_1,\dots,\xi_n),t\big)$ to $\big((\hat{\xi}_1,\dots,\hat{\xi}_n),t^{l_{\Gamma}(a)}\big)$, for $\gamma=\gamma_a$, with the $\mathbb{G}$-action on $g_{\gamma,I}^{-1}(0)$ given by $s\cdot(\xi_1,\dots,\xi_n)=(s^{a_1}\xi_1,\dots,s^{a_n}\xi_n)$ and the $\mathbb{G}$-action on $\mathbb{G}$ given by the multiplicative translation, and $g_{\gamma}^{-1}(0)\times_{\kappa} \mathbb{G}\rightarrow X_0(g)$ being $\mathbb{G}$-equivariant, $g_{\gamma}^{-1}(0)\times_{\kappa}\mathbb{G}\to\mathbb{G}$ being diagonally monomial of weight $l_{\Gamma}(a)$ with respect to the $\mathbb{G}$-action.
\begin{lemma}\label{lem4.1}
The following formulas hold in $\mathscr{M}_{X_0(g)\times_{\kappa} \mathbb{G}}^{\mathbb{G}}$ for every $a$ in $\sigma_{\gamma,I}$:
\begin{enumerate}
\item[(i)] If there is a non-empty subset $I$ of $\{1,\dots,n\}$ such that $a_i>m$ for any $i\in I$ and $g|_{\mathbb{A}_{\kappa}^{I^c}}=0$, then $[\mathscr{X}_{a,m}(g)]=0$.
\end{enumerate}
If $a_i\leq l_{\Gamma}(a)$ for any $i=1,\dots,n$, we have
\begin{enumerate}
\item[(ii)] $[\mathscr{X}_{a,l_{\Gamma}(a)}(g)]=\Phi_{\gamma,I}\mathbb{L}^{nl_{\Gamma}(a)-s(a)}$,
\item[(iii)] $[\mathscr{X}_{a,l_{\Gamma}(a)+k}(g)]=\Psi_{\gamma,I}\mathbb{L}^{n(l_{\Gamma}(a)+k)-s(a)}$ for $k\geq 1$.
\end{enumerate}
Here, $\mathbb{A}_{\kappa}^{I^c}:=\{(x_1,\dots,x_n)\in\mathbb{A}_{\kappa}^n\ |\ x_i=0\ \forall i\in I\}$, and $s(a):=\sum_{i=1}^na_i$.
\end{lemma}
\begin{proof}
The item (i) follows from the definition of $\mathscr{X}_{a,m}(g)$ and from the hypothesis on $g$. Indeed, every element of $\pi_m(\mathscr{X}_a)$ has the form $\varphi(t)=(x_1(t),\dots,x_n(t))$, where $x_j(t)$ is a polynomial of degree $\leq m$ in a variable $t$ for any $j=1,\dots,n$, and $x_i(t)$ is the zero polynomial if $i$ in $I$. Then $g(\varphi(t))=0$ and $ord_tg(\varphi)=\infty$, which means that $\mathscr{X}_{a,m}(g)=\emptyset$.\\
\indent The items (ii) and (iii) may be deduced from proofs of Guibert in \cite{G} (cf. \cite{G}, Lemmas 2.1.1 and 2.1.2) and from the isomorphism $\mathscr{M}_{X_0(g)}^{\hat{\mu}}\cong\mathscr{M}_{X_0(g)\times\mathbb{G}}^{\mathbb{G}}$ (cf. \cite{GLM1}, Proposition 2.6). In \cite{G}, Section 2.1 (in particular, Lemmas 2.1.1 and 2.1.2), Guibert only considers functions of the form $\sum_{\alpha\in\mathbb{N}_{>0}^n}f_{\alpha}x^{\alpha}$. Observe that his condition ``$\alpha\in\mathbb{N}_{>0}^n$" is equivalent to that $a_i\leq l_{\Gamma}(a)$ for any $i=1,\dots,n$. Finally, notice that the hypothesis of the non-degeneracy with respect to $\Gamma$ is in fact the main tool for the proofs.\\
 \indent There is also a way to prove (ii) directly as follows. An element $\varphi(t)$ of $\mathscr{X}_{a,l_{\Gamma}(a)}(g)$ has the form $\varphi(t)=(x_1(t),\dots,x_n(t))$, where $x_i(t)=\sum_{m=a_i}^{l_{\Gamma}(a)}c_{i,m}t^m$ with $c_{i,a_i}\not=0$ for $i=1,\dots,n$. Note that the coefficient of $t^{l_{\Gamma}(a)}$ in $g(\varphi(t))$ is equal to
\begin{align*}
\frac{1}{l_{\Gamma}(a)!}\cdot\frac{d^{l_{\Gamma}(a)}g(\varphi(t))}{dt^{l_{\Gamma}(a)}}|_{t=0}&=\frac{1}{l_{\Gamma}(a)!}\cdot\frac{d^{l_{\Gamma}(a)}g_{\gamma}(\varphi(t))}{dt^{l_{\Gamma}(a)}}|_{t=0}\\
&= g_{\gamma}(c_{1,a_1},\dots,c_{n,a_n})
\end{align*}
which is non-zero for every $a$ in $\sigma_{\gamma,I}$ and $(c_{1,a_1},\dots,c_{n,a_n})$ in $X_{\gamma}$. One deduces from this that $\mathscr{X}_{a,l_{\Gamma}(a)}(g)$ is isomophic to $X_{\gamma}\times_{\kappa}\mathbb{A}_{\kappa}^{nl_{\Gamma}(a)-s(a)}$ via the map 
$$\varphi(t)\mapsto \big((c_{i,a_i})_{1\leq i\leq n},(c_{i,m})_{1\leq i\leq n,a_i+1\leq m\leq l_{\Gamma}(a)}\big).$$
Here the action of $\mathbb{G}$ on $\mathbb{A}_{\kappa}^1$ is trivial. For any $s$ in $\mathbb{G}$, the arc $\varphi(st)$ is mapped to
$$\big((s^{a_i}c_{i,a_i})_{1\leq i\leq n},(c_{i,m})_{1\leq i\leq n,a_i+1\leq m\leq l_{\Gamma}(a)}\big)$$
which is by definition equal to 
$$s\cdot\big((c_{i,a_i})_{1\leq i\leq n},(c_{i,m})_{1\leq i\leq n,a_i+1\leq m\leq l_{\Gamma}(a)}\big).$$
This means that the $\mathbb{G}$-action is compatible with the isomorphism, i.e the isomorphism is $\mathbb{G}$-equivariant. Then the item (ii) follows.
\end{proof}
\begin{remark}\label{remarksss}
We have not known yet how to compute $[\mathscr{X}_{a,l_{\Gamma}(a)+k}(g)]$, for $k\geq 0$, without the assumptions as in Lemma \ref{lem4.1}.
\end{remark}
\begin{remark}\label{remark4.2}
Lemma \ref{lem4.1} and Remark \ref{remarksss} explain the reason why in the rest of this paper we will always assume that no vertex of the Newton polyhedron $\Gamma$ of $g$ lies in a coordinate plane, i.e. $a_i\leq l_{\Gamma}(a)$ for any $i=1,\dots,n$. In this case, $l_{\Gamma}(a)$ is expressed as $\sum_{i=1}^n\alpha_ia_i$ with $\alpha_i>0$ for any $i=1,\dots,n$. By Lemma \ref{lem4.1}, this hypothesis guarantees that, for every compact face $\gamma$ of $\Gamma$, $I$ in $\mathfrak{S}_{\gamma}$, and $k\geq 0$, all the terms of the sum $\sum_{a\in\sigma_{\gamma,I}}[\mathscr{X}_{a,l_{\Gamma}(a)+k}(g)]$ are non-zero if $\Phi_{\gamma,I}$ (resp. $\Psi_{\gamma,I}$) is non-zero (i.e. if the sum is non-zero). For the purpose of our work, we would like to consider such sums which may be reduced to the case of Lemma \ref{lem2.1}.
\end{remark}

\subsection{An explicit formula for $i_1^*\mathcal{S}_g$}
Assume that $g$ is a regular function on $\mathbb{A}_{\kappa}^n$ non-degenerate with respect to its Newton polyhedron $\Gamma$, that no vertex of $\Gamma$ lies in a coordinate $m$-plane ($m=1,\dots,n-1$), and that $X_0(g)$ contains $\mathbb{A}_{\kappa}^{n_1}\times_{\kappa}\{0\}$. One then deduces from Remark \ref{remark4.2} and Lemma \ref{lem4.1} that 
$$i_1^*Z^0(T)=\sum_{\gamma\in\Gamma_c}\sum_{I\in\mathfrak{S}_{\gamma}}\sum_{a\in \sigma_{\gamma,I}}i_1^*\Phi_{\gamma,I}\mathbb{L}^{-s(a)}T^{l_{\Gamma}(a)},$$
and
\begin{align*}
i_1^*Z^1(T)=&\ i_1^*\Big(\sum_{\gamma\in\Gamma_c}\sum_{I\in\mathfrak{S}_{\gamma}}\sum_{a\in \sigma_{\gamma,I}}\Psi_{\gamma,I}\mathbb{L}^{-s(a)}T^{l_{\Gamma}(a)}\sum_{k\geq 1}\mathbb{L}^{-k}T^k\Big)\\
=&\ \frac{\mathbb{L}^{-1}T}{1-\mathbb{L}^{-1}T}\sum_{\gamma\in\Gamma_c}\sum_{I\in\mathfrak{S}_{\gamma}}\sum_{a\in \sigma_{\gamma,I}}i_1^*\Psi_{\gamma,I}\mathbb{L}^{-s(a)}T^{l_{\Gamma}(a)}.
\end{align*}
\begin{proposition}\label{prop4.2}
With the previous notations and hypotheses, the following formula holds in $\mathscr{M}_{\mathbb{A}_{\kappa}^{n_1}\times_{\kappa} \mathbb{G}}^{\mathbb{G}}$:
\begin{align*}
i_1^*\mathcal{S}_g= \sum_{\gamma\in\Gamma_c}(-1)^{n+1-\dim(\gamma)}\sum_{I\in\mathfrak{S}_{\gamma}}(-1)^{|I|}[\mathbb{A}_{\kappa}^{n_1}\times_{X_0(g)}(\Phi_{\gamma,I}-\Psi_{\gamma,I})].
\end{align*}
\end{proposition}
\begin{proof}
The positivity of the sum function $s$ on $\overline{\sigma_{\gamma,I}}\setminus\{0\}$ is evident, that of the function $l_{\Gamma}$ on $\overline{\sigma_{\gamma,I}}\setminus\{0\}$ follows straightforward from Remark \ref{remark4.2}. Applying Lemma \ref{lem2.1}, notice that $\dim(\sigma_{\gamma,I})=n-|I|-\dim(\gamma)$, we have
\begin{align*}
\lim_{T\rightarrow\infty}\sum_{a\in \sigma_{\gamma,I}}\Phi_{\gamma,I}\mathbb{L}^{-s(a)}T^{l_{\Gamma}(a)}&=\Phi_{\gamma,I}\lim_{T\rightarrow\infty}\sum_{a\in \sigma_{\gamma,I}}\mathbb{L}^{-s(a)}T^{l_{\Gamma}(a)}\\
&=(-1)^{n-|I|-\dim(\gamma)}\Phi_{\gamma,I},
\end{align*}
and
\begin{align*}
\lim_{T\rightarrow\infty}\sum_{a\in \sigma_{\gamma,I}}\Psi_{\gamma,I}\mathbb{L}^{-s(a)}T^{l_{\Gamma}(a)}&=\Psi_{\gamma,I}\lim_{T\rightarrow\infty}\sum_{a\in \sigma_{\gamma,I}}\mathbb{L}^{-s(a)}T^{l_{\Gamma}(a)}\\
&=(-1)^{n-|I|-\dim(\gamma)}\Psi_{\gamma,I}.
\end{align*}
It follows that
\begin{align*}
\lim_{T\rightarrow\infty}i_1^*Z^0(T)=\sum_{\gamma\in\Gamma_c}(-1)^{n-\dim(\gamma)}\sum_{I\in\mathfrak{S}_{\gamma}}(-1)^{|I|}i_1^*\Phi_{\gamma,I},
\end{align*}
and
\begin{align*}
\lim_{T\rightarrow\infty}i_1^*Z^1(T)=\sum_{\gamma\in\Gamma_c}(-1)^{n+1-\dim(\gamma)}\sum_{I\in\mathfrak{S}_{\gamma}}(-1)^{|I|}i_1^*\Psi_{\gamma,I}.
\end{align*}
Then the proposition is proved.
\end{proof}
\begin{example}[cf. Example \ref{ex3.8}]
In the case $\Gamma_g$ has a unique compact face $P$, the classes $\Psi_{P,I}$ vanish. If we assume that $\alpha_i>0$ for every $i=1,\dots,n$, we have
$$i_1^*\mathcal{S}_g= (-1)^{n+1}\sum_{I\subset\{1,\dots,n_1\}}(-1)^{|I|}[\mathbb{A}_{\kappa}^{n_1}\times_{X_0(g)}\Phi_{P,I}].$$
\end{example}
\begin{corollary}[\cite{G}]\label{cor4.5}
Assume that $g$ is given by $g(x)=\sum_{\alpha\in\mathbb{N}_{>0}^n}a_{\alpha}x^{\alpha}$ in $\kappa[x]$ with $g(0)=0$. If $g$ is non-degenerate with respect to $\Gamma$, then 
\begin{align*}
\mathcal{S}_{g,0}= (-1)^{n-1}\sum_{\gamma\in\Gamma_c}(-1)^{\dim(\gamma)}[\{0\}\times_{X_0(g)}(\Phi_{\gamma,I}-\Psi_{\gamma,I})]
\end{align*}
holds in $\mathscr{M}_{\mathbb{G}}^{\mathbb{G}}$.
\end{corollary}
\begin{proof}
(See Remark \ref{rem3.10}) Apply Proposition \ref{prop4.2} to the case $n_1=0$. Here the natural inclusion $i_1:\mathbb{A}_{\kappa}^{n_1}\hookrightarrow\mathbb{A}_{\kappa}^n$ reduces to the inclusion $i_0: \{0\}\hookrightarrow\mathbb{A}_{\kappa}^n$. Moreover, in this case, by Lemma \ref{lem3.8}, for every compact face $\gamma$ of $\Gamma$, we have $\mathfrak{S}_{\gamma}=\{\emptyset\}$. Thus this corollary follows. Observe that this formula was already obtained by Guibert (cf. \cite{G}, Proposition 2.1.6).
\end{proof}

\subsection{}
Consider the function $g(x)=\sum_{\alpha\in H\cap\mathbb{N}^n}a_{\alpha}x^{\alpha}$ on $\mathbb{A}_{\kappa}^n$, where $H$ is the hyperplane in $\mathbb{R}_{\geq 0}^n$ defined by the following equation
\begin{align*}
\alpha_1+\cdots+\alpha_{n_1}=\alpha_{n_1+1}+\cdots+\alpha_p
\end{align*}
for some fixed $p$ such that $n_1< p\leq n$. Here as well as in Corollary \ref{cor4.5} we denote $x^{\alpha}$ for $x_1^{\alpha_1}\cdots x_n^{\alpha_n}$ with $\alpha=(\alpha_1,\dots,\alpha_n)$. Because $supp(g)$ lies on the hyperplane $H$, the compact faces of $\Gamma$ are contained in $H$. Moreover, for the same reason, for each compact $\gamma$, the non-compact faces of $\Gamma$ leaning on $\gamma$ exist. Note that, in this case, $\mathbb{A}_{\kappa}^{n_1}$ is naturally viewed as a subset of $X_0(g)$.
\begin{lemma}\label{lem4.4}
Assume that $g(x)=\sum_{\alpha\in H\cap\mathbb{N}^n}a_{\alpha}x^{\alpha}$ is non-degenerate with respect to $\Gamma$. Then, for every compact face $\gamma$ of $\Gamma$, we have $|\mathfrak{M}_{\gamma}|=1$ and the unique element of $\mathfrak{M}_{\gamma}$ is nonempty.
\end{lemma}
\begin{proof}
Let $\gamma$ be a compact face of $\Gamma$. Assume that $\gamma+\mathbb{R}_{\geq 0}^I$ is a face of $\Gamma$. Then, by Lemma \ref{lem3.8}, the cone $\sigma_{\gamma,I}$ is contained in $\mathbb{R}_{\geq 0}^{n_1}\times\mathbb{R}_{>0}^{n_2}$ if and only if $I$ is contained in $\{1,\dots,n_1\}$. Furthermore, we claim that if $\gamma+\mathbb{R}_{\geq 0}^I$ and $\gamma+\mathbb{R}_{\geq 0}^J$ are faces leaning on $\gamma$ such that the corresponding cones $\sigma_{\gamma,I}$ and $\sigma_{\gamma,J}$ are both contained in $\mathbb{R}_{\geq 0}^{n_1}\times\mathbb{R}_{>0}^{n_2}$, then so is $\gamma+\mathbb{R}_{\geq 0}^{I\cup J}$. Indeed, since $(\alpha_1,\dots,\alpha_n)$ is in $H$ one deduces that if $I$ and $J$ are contained in $\{1,\dots,n_1\}$, the intersection of $\gamma+\mathbb{R}_{\geq 0}^{I\cup J}$ with the interior of $\Gamma$ is empty. This together with the fact that $\gamma+\mathbb{R}_{\geq 0}^I$ and $\gamma+\mathbb{R}_{\geq 0}^J$ are faces of $\Gamma$ show that $\gamma+\mathbb{R}_{\geq 0}^{I\cup J}$ is a face of $\Gamma$ leaning on $\gamma$ such that $\sigma_{\gamma,I\cup J}$ is contained in $\mathbb{R}_{\geq 0}^{n_1}\times\mathbb{R}_{>0}^{n_2}$. \\
\indent As a consequence of the above claim, for each compact face $\gamma$ of $\Gamma$, there exists a unique maximal subset $M$ of $\{1,\dots,n_1\}$ such that $\gamma+\mathbb{R}_{\geq 0}^M$ is a face of $\Gamma$, which leans on $\gamma$, and $\sigma_{\gamma,M}$ is contained in $\mathbb{R}_{\geq 0}^{n_1}\times\mathbb{R}_{>0}^{n_2}$. The nonemptyness of the set $M$ follows from the fact that $supp(g)$ lies on the hyperplane $H$. 
\end{proof}
\begin{proposition}\label{prop4.5}
Assume that $g(x)=\sum_{\alpha\in H\cap\mathbb{N}_{>0}^n}a_{\alpha}x^{\alpha}$ is non-degenerate with respect to $\Gamma$. Then $\int_{\mathbb{A}_{\kappa}^{d_1}}i_1^*\mathcal{S}_g$ vanishes in $\mathscr{M}_{\mathbb{G}}^{\mathbb{G}}$.
\end{proposition}
\begin{proof}
Let $\gamma$ be a compact face of $\Gamma$. By Lemma \ref{lem4.4}, the set $\mathfrak{M}_{\gamma}$ has a unique element and this element is nonempty. Assume $\mathfrak{M}_{\gamma}=\{M\}$ with $|M|\geq 1$. Note that $A_{\gamma}=\int_{\mathbb{A}_{\kappa}^{d_1}}i_1^*\Phi_{\gamma,I}$ and $B_{\gamma}=\int_{\mathbb{A}_{\kappa}^{d_1}}i_1^*\Psi_{\gamma,I}$ depend only on $\gamma$, not on $I$ contained in $M$. Because of the fact that, if $m\geq 1$, $\sum_{j=0}^m(-1)^j\big(_j^m\big)=0$, one deduces that
$$\sum_{I\subset M}(-1)^{|I|}(A_{\gamma}-B_{\gamma})=0.$$
A hypothesis on $g$, namely $\alpha \in H\cap\mathbb{N}_{>0}^n$, means that no vertex of the Newton polyhedron $\Gamma$ of $g$ lies in a coordinate plane. By Proposition \ref{prop4.2}, the image $\int_{\mathbb{A}_{\kappa}^{d_1}}i_1^*\mathcal{S}_g$ of $\mathcal{S}_g$ vanishes in $\mathscr{M}_{\mathbb{G}}^{\mathbb{G}}$.
\end{proof}


\section{Kontsevich-Soibelman's conjecture}\label{sec5}
In this section, we will show that, under certain assumptions, Conjecture \ref{conj} is true.
\subsection{Composition with a polynomial in two variables}
We consider the conjecture of Kontsevich and Soibelman (Conjecture \ref{conj}) in the case where $F$ has the form $F(x,y,z)=f(g_1(x,y),g_2(z))$, where $f$ is a polynomial in two variables with $f(0,y)$ non-zero of positive degree, $g_1$ is a function on $\mathbb{A}_{\kappa}^{d_1}\times_{\kappa}\mathbb{A}_{\kappa}^{d_2}$ such that $g_1(tx,t^{-1}y)=g_1(x,y)$, $g_1(0,0)=0$, and $g_2$ is a regular function on $\mathbb{A}_{\kappa}^{d_3}$. 
Denote $\bold{g}=g_1\times g_2$ and $X_0(\bold{g})=\{(x,y,z)\ |\ g_1(x,y)=g_2(z)=0\}$. In particular, $X_0(\bold{g})$ contains $\mathbb{A}_k^{d_1}\times\{0\}$. We denote by $i_1$ the inclusion of $\mathbb{A}_{\kappa}^{d_1}\times_{\kappa}\mathbb{G}$ into $X_0(f\circ \bold{g})\times_{\kappa}\mathbb{G}$. Recall that, in this case, $h(z)=f(0,g_2(z))$.
\begin{theorem}\label{thm5.1}
Assume that $f$ is a polynomial in two variables with $f(0,y)$ non-zero of positive degree. Let $g_1$ be a regular function on $\mathbb{A}_{\kappa}^{d_1}\times_{\kappa}\mathbb{A}_{\kappa}^{d_2}$ non-degenerate with respect to its Newton polyhedron $\Gamma_{g_1}$ such that $g_1(0,0)=0$, no vertex of  $\Gamma_{g_1}$ lies in a coordinate plane, and $g_1(tx,t^{-1}y)=g_1(x,y)$ for every $t$ in $\mathbb{G}$. Let $g_2$ be a regular function on $\mathbb{A}_{\kappa}^{d_3}$. Then, the following formula
$$\int_{\mathbb{A}_{\kappa}^{d_1}}i_1^*\mathcal{S}_{f\circ \bold{g}}=\mathbb{L}^{d_1}\mathcal{S}_{h,0}$$
holds in $\mathscr{M}_{\mathbb{G}}^{\mathbb{G}}$. In other words, in this case, Conjecture \ref{conj} is true.
\end{theorem}
\begin{proof}
In \cite{GLM3}, Guibert, Loeser and Merle consider the motivic Milnor fiber of a composition of the form $f(g_1,g_2)$ where $g_1$ and $g_2$ have no variable in common and $f$ is a polynomial in $\kappa[x,y]$ such that $f(0,y)$ is non-zero of positive degree. To describe it, they used the generalized convolution operators $\Psi_Q$ defined in \cite{GLM2} and the tree of contact $\tau(f,0)$ constructed in terms of Puiseux expansions by Guibert (\cite{G}), here $0$ is the origin of $\mathbb{A}_{\kappa}^{d}$ with $d=d_1+d_2+d_3$. To any rupture vertex $v$ of $\tau(f,0)$ one attaches a weighted homogeneous polynomial $Q_{v}$ in $\kappa[X,Y]$. The virtual objects $A_v$ are defined inductively in terms of the tree of contact $\tau(f,0)$ and $A_{v_0}$, where $v_0$ is the first (extended) rupture vertex of the tree and $A_{v_0}$ depends only on $g$. Let $i$ be the inclusion of $X_0(\bold{g})\times_{\kappa} \mathbb{G}$ into $X_0(f\circ \bold{g})\times_{\kappa} \mathbb{G}$. Let $m_0$ be the order of $0$ as a root of $f(0,y)$. By the main theorem of \cite{GLM3}, the following formula
$$i^*\mathcal{S}_{f\circ \bold{g}}=\mathcal{S}_{g_2^{m_0}}([X_0(g_1)])-\sum_v\Psi_{Q_v}(A_v)$$
holds in $\mathscr{M}_{X_0(\bold{g})\times_{\kappa} \mathbb{G}}^{\mathbb{G}}$, where $\Psi_{Q_v}$ denotes the convolution defined in \cite{GLM3} and the sum runs over the augmented set of rupture vertices of the tree $\tau(f,0)$. The $i_1$ in the theorem is the inclusion of $\mathbb{A}_{\kappa}^{d_1}\times_{\kappa}\mathbb{G}$ into $X_0(f\circ \bold{g})\times_{\kappa}\mathbb{G}$, but by abuse of notation, we also use $i_1$ for the inclusion $\mathbb{A}_{\kappa}^{d_1}\times_{\kappa}\mathbb{G}\hookrightarrow X_0(\bold{g})\times_{\kappa}\mathbb{G}$. Thus $i_1$ and $i\circ i_1$ are in fact the same thing. Take the operator $\int_{\mathbb{A}_{\kappa}^{d_1}}i_1^*$ for two sides of the previous formula, we have
$$\int_{\mathbb{A}_{\kappa}^{d_1}}i_1^*\mathcal{S}_{f\circ \bold{g}}=\int_{\mathbb{A}_{\kappa}^{d_1}}i_1^*\mathcal{S}_{g_2^{m_0}}([X_0(g_1)])-\sum_v\int_{\mathbb{A}_{\kappa}^{d_1}}i_1^*\Psi_{Q_v}(A_v).$$
\indent We claim that, with previous notations and hypotheses, the formula
$$\int_{\mathbb{A}_{\kappa}^{d_1}}i_1^*\mathcal{S}_{g_2^{m_0}}([X_0(g_1)])=\mathbb{L}^{d_1}\mathcal{S}_{h,0}$$
holds in $\mathscr{M}_{\mathbb{G}}^{\mathbb{G}}$. Indeed, as in \cite{GLM1}, proof of Theorem 5.18, one can check that 
$$i^*\mathcal{S}_{g_2^{m_0}}([X_0(g_1)])=[g_1^{-1}(0)]\boxtimes\mathcal{S}_{g_2^{m_0}}.$$
By the hypotheses on $g_1$ and the fact that $i_1(\mathbb{A}_{\kappa}^{d_1})\cap g_2^{-1}(0)=\{0\}$, we have $i_1^*[g_1^{-1}(0)]=[\mathbb{A}_{\kappa}^{d_1}]=\mathbb{L}^{d_1}$ and $i_1^*\mathcal{S}_{g_2^{m_0}}=i_0^*\mathcal{S}_{g_2^{m_0}}=\mathcal{S}_{g_2^{m_0},0}$. One deduces that 
$$i_1^*\mathcal{S}_{g_2^{m_0}}([X_0(g_1)])=i_1^*\Big([g_1^{-1}(0)]\boxtimes\mathcal{S}_{g_2^{m_0}}\Big)=\mathbb{L}^{d_1}\mathcal{S}_{g_2^{m_0},0}.$$
By definition of $h$ and $m_0$, $\mathcal{S}_{g_2^{m_0},0}=\mathcal{S}_{h,0}$, the claim then follows. So, in order to finish the proof of Theorem \ref{thm5.1}, it suffices to prove that $\int_{\mathbb{A}_{\kappa}^{d_1}}i_1^*\Psi_{Q_v}(A_v)=0$ for every (extended) rupture vertex $v$ of $\tau(f,0)$.\\
\indent Let $v_0$ be the first (extended) rupture vertex of the tree of contact $\tau(f,p)$. As in \cite{GLM3}, the virtual object $A_{v_0}$ in $\mathscr{M}_{X_0(\bold{g})\times_{\kappa}(\mathbb{A}_{\kappa}^1\times_{\kappa} \mathbb{G})}^{\mathbb{G}}$ is defined by $A_{v_0}:=\mathcal{S}'_{g_2}\boxtimes\mathcal{S}_{g_1}$, where $\mathcal{S}'_{g_2}$ is an element in $\mathscr{M}_{X_0(g_2)\times_{\kappa}\mathbb{A}_{\kappa}^1}^{\mathbb{G}}$ which is the ``disjoint sum'' of $\mathcal{S}_{g_2}$ in $\mathscr{M}_{X_0(g_2)\times_{\kappa} \mathbb{G}}^{\mathbb{G}}$ and $X_0(g_2)$ in $\mathscr{M}_{X_0(g_2)}$.
\begin{lemma}\label{lem5.2}
Assume that $g_1$ is a regular function on $\mathbb{A}_{\kappa}^{d_1}\times_{\kappa}\mathbb{A}_{\kappa}^{d_2}$ non-degenerate with respect to its Newton polyhedron $\Gamma_{g_1}$ such that $g_1(0,0)=0$, no vertex of  $\Gamma_{g_1}$ lies in a coordinate plane, and $g_1(tx,t^{-1}y)=g_1(x,y)$ for every $t$ in $\mathbb{G}$. Let $g_2$ be a regular function on $\mathbb{A}_{\kappa}^{d_3}$. Then $\int_{\mathbb{A}_{\kappa}^{d_1}}i_1^*\Psi_Q(A_{v_0})$ vanishes in $\mathscr{M}_{\mathbb{G}}^{\mathbb{G}}$ for every quasi-homogeneous polynomial $Q$.
\end{lemma}
\begin{proof}
The asumptions on $g_1$ mean that we can write $g_1$ in the form
$$g_1(x,y)=\sum_{(\alpha,\beta)\in H\cap\mathbb{N}_{>0}^{d_1+d_2}}a_{\alpha\beta}x_1^{\alpha_1}\cdots x_{d_1}^{\alpha_{d_1}}y_1^{\beta_1}\cdots y_{d_2}^{\beta_{d_2}},$$
where $H$ is given by $\alpha_1+\cdots+\alpha_{d_1}=\beta_1+\cdots+\beta_{d_2}$. By Proposition \ref{prop4.5}, $\int_{\mathbb{A}_{\kappa}^{d_1}}i_1^*\mathcal{S}_{g_1}$ vanishes in $\mathscr{M}_{\mathbb{G}}^{\mathbb{G}}$, hence $\int_{\mathbb{A}_{\kappa}^{d_1}}i_1^*A_{v_0}$ vanishes in $\mathscr{M}_{\mathbb{A}_{\kappa}^1\times_{\kappa} \mathbb{G}}^{\mathbb{G}}$. Notice that here the $i_1$ is once again abused to denote the natural inclusion $\mathbb{A}_{\kappa}^{d_1}\times_{\kappa}\mathbb{A}_{\kappa}^1\times_{\kappa}\mathbb{G}\hookrightarrow X_0(\bold{g})\times_{\kappa}\mathbb{A}_{\kappa}^1\times_{\kappa}\mathbb{G}$. Because the following diagram
$$
\begin{CD}
\mathscr{M}_{X_0(\bold{g})\times_{\kappa}\mathbb{A}_{\kappa}^1\times_{\kappa} \mathbb{G}}^{\mathbb{G}}@>\Psi_Q>>\mathscr{M}_{X_0(\bold{g})\times_{\kappa} \mathbb{G}}^{\mathbb{G}}\\
@V\int_{\mathbb{A}_{\kappa}^{d_1}}i_1^*VV @V\int_{\mathbb{A}_{\kappa}^{d_1}}i_1^*VV\\
\mathscr{M}_{\mathbb{A}_{\kappa}^1\times_{\kappa} \mathbb{G}}^{\mathbb{G}}@>\Psi_Q>>\mathscr{M}_{\mathbb{G}}^{\mathbb{G}}\\
\end{CD}
$$
 commutes, the lemma thus follows. 
\end{proof}
\indent Let $v$ be an abitrary rupture vertex of the tree of contact $\tau(f,0)$ and $a(v)$ the predecessor of $v$ in the augmented set of rupture vertices. Then the polynomial $Q_v$ is a factor of $Q_{a(v)}$. Suppose that $Q_v(X,1)$ has $m_v$ disjoint zeroes in $\mathbb{A}_{\kappa}^1$.
\begin{lemma}\label{lem5.3} 
The equality $A_v=m_vA_{a(v)}$ holds in $\mathscr{M}_{X_0(\bold{g})\times_{\kappa}\mathbb{A}_{\kappa}^1\times_{\kappa} \mathbb{G}}^{\mathbb{G}}$.
\end{lemma}
\begin{proof}
We first notice that $Q_v^{-1}(0)$ is a smooth subvariety in $\mathbb{G}\times_{\kappa} \mathbb{G}$, equivariant under a diagonal $\mathbb{G}$-action and that the second projection $pr_2$ of the product $\mathbb{A}_{\kappa}^1\times_{\kappa} \mathbb{G}$ induces a homogeneous fibration $Q_v^{-1}(0)\rightarrow \mathbb{G}$. We denote by $B_v$ the restriction of $A_{a(v)}$ above $Q_v^{-1}(0)$. Then, by \cite{GLM3}, the element $A_v$ in $\mathscr{M}_{X_0(\bold{g})\times_{\kappa}\mathbb{A}_{\kappa}^1\times_{\kappa} \mathbb{G}}^{\mathbb{G}}$ is defined as the external product of the class of $id:\mathbb{A}_{\kappa}^1\rightarrow\mathbb{A}_{\kappa}^1$ by the induced map $pr_2: B_v\rightarrow \mathbb{G}$, which is diagonally monomial when restricted to $X_0(\bold{g})\times_{\kappa} \mathbb{G}\times_{\kappa} \mathbb{G}$. \\
\indent Consider the fibration $pr_2: B_v\rightarrow \mathbb{G}$ defined by the composition of $B_v\rightarrow Q_v^{-1}(0)$ and $pr_2: Q_v^{-1}(0)\rightarrow \mathbb{G}$. Then each fiber of $pr_2: B_v\rightarrow \mathbb{G}$ is a disjoint union of $m_v$ copies of a fiber of $A_{a(v)}\rightarrow \mathbb{A}_{\kappa}^1\times_{\kappa} \mathbb{G}$ over one point $(a,b)$ in $\mathbb{A}_{\kappa}^1\times_{\kappa} \mathbb{G}$. It follows that $A_v=m_vA_{a(v)}$ in $\mathscr{M}_{X_0(\bold{g})\times_{\kappa}(\mathbb{A}_{\kappa}^1\times_{\kappa} \mathbb{G})}^{\mathbb{G}}$.
\end{proof}
\indent It follows from Lemma \ref{lem5.2} and Lemma \ref{lem5.3} that $\int_{\mathbb{A}_{\kappa}^{d_1}}i_1^*\Psi_{Q_v}(A_v)=0$ for every (extended) rupture vertex $v$ of $\tau(f,0)$. This completes the proof of Theorem \ref{thm5.1}.
\end{proof}
\begin{remark}
In the case $f(x,y)=x+y$, the result can also be obtained directly from the Motivic Thom-Sebastiani Theorem (cf. \cite{DL4,DL3}).
\end{remark}

\subsection{}
In the following proposition, we prove the conjecture of Kontsevich and Soibelman under some other conditions on $F=g$, namely assuming $F$ is non-degenerate with respect to its Newton polyhedron $\Gamma$ and no vertex of $\Gamma$ lies in a coordinate plane.
\begin{proposition}\label{prop5.4}
Let $g$ be a regular function on $\mathbb{A}_{\kappa}^{d_1}\times_{\kappa}\mathbb{A}_{\kappa}^{d_2}\times_{\kappa}\mathbb{A}_{\kappa}^{d_3}$ such that $g(0,0,z)=0$ for every $z$ in $\mathbb{A}_{\kappa}^{d_3}$, and $g(tx,t^{-1}y,z)=g(x,y,z)$ for every $t$ in $\mathbb{G}$ and $(x,y,z)$ in $\mathbb{A}_{\kappa}^{d_1}\times_{\kappa}\mathbb{A}_{\kappa}^{d_2}\times_{\kappa}\mathbb{A}_{\kappa}^{d_3}$. If $g$ is non-degenerate with respect to its Newton polyhedron $\Gamma$ and no vertex of $\Gamma$ lies in a coordinate plane, then $\int_{\mathbb{A}_{\kappa}^{d_1}}i_1^*\mathcal{S}_g$ vanishes in $\mathscr{M}_{\mathbb{G}}^{\mathbb{G}}$. In other words, Conjecture \ref{conj} is true in this case.
\end{proposition}
\begin{proof}
Write the function $g$ in the following form
$$g(x,y,z)=\sum_{(a,b,c)\in H\cap\mathbb{N}^d_{>0}}g_{a,b,c}x^ay^bz^c,$$
where $d=d_1+d_2+d_3$ and $H$ is given by the equation $a_1+\cdots+a_{d_1}=b_1+\cdots+b_{d_2}$. By Proposition \ref{prop4.5}, $\int_{\mathbb{A}_{\kappa}^{d_1}}i_1^*\mathcal{S}_g$ vanishes in $\mathscr{M}_{\mathbb{G}}^{\mathbb{G}}$. Notice that, in this case, $h(z)=F(0,0,z)=g(0,0,z)=0$, hence $\mathcal{S}_{h,0}$ also vanishes in $\mathscr{M}_{\mathbb{G}}^{\mathbb{G}}$.
\end{proof}

\subsection{Functions of Steenbrink type}
We consider now the case $F(x,y,z)=g(x,y,z)+h(z)^{\ell}$, where $g$ is as in Proposition \ref{prop5.4}, $h(z)$ is regular on $\mathbb{A}_{\kappa}^{d_3}$ such that $h(0)=0$, and $\ell$ is a large enough natural number. By composition with the projection, we will view $h$ as a function on $\mathbb{A}_{\kappa}^d$.
\begin{theorem}\label{thm5.6}
Let $F(x,y,z)=g(x,y,z)+h(z)^{\ell}$, where $g$ is as in Proposition \ref{prop5.4}, $h(z)$ is regular on $\mathbb{A}_{\kappa}^{d_3}$ such that $h(0)=0$, $\ell$ is a natural number. There exists a positive real number $N$ such that if $\ell>N$, the following formula holds in $\mathscr{M}_{\mathbb{G}}^{\mathbb{G}}$ :
$$\int_{\mathbb{A}_{\kappa}^{d_1}}i_1^*\mathcal{S}_F=\mathbb{L}^{d_1}\mathcal{S}_{h^{\ell},0}.$$
\end{theorem}
\begin{proof}
Let us denote by $i$ and $j$ the inclusion of $(X_0(g)\cap X_0(h))\times_{\kappa} \mathbb{G}$ in $X_0(g)\times_{\kappa} \mathbb{G}$ and $X_0(F)\times_{\kappa} \mathbb{G}$, respectively. The existence of $N$ is shown by \cite{GLM1}, Theorem 5.7, and also by this theorem, for $\ell>N$, we have
$$j^*\mathcal{S}_F-i^*\mathcal{S}_g=\mathcal{S}_{h^{\ell}}([X_0(g)])-\Psi_{\Sigma}(\mathcal{S}_{h^{\ell}}(\mathcal{S}_g)),$$
where $\Psi_{\Sigma}$ is the convolution defined in \cite{GLM1}. Then we get
$$\int_{\mathbb{A}_{\kappa}^{d_1}}i_1^*\mathcal{S}_F-\int_{\mathbb{A}_{\kappa}^{d_1}}i_1^*\mathcal{S}_g=\int_{\mathbb{A}_{\kappa}^{d_1}}i_1^*\mathcal{S}_{h^{\ell}}([X_0(g)])-\int_{\mathbb{A}_{\kappa}^{d_1}}i_1^*\Psi_{\Sigma}(\mathcal{S}_{h^{\ell}}(\mathcal{S}_g)).$$
Now, by Proposition \ref{prop5.4}, $\int_{\mathbb{A}_{\kappa}^{d_1}}i_1^*\mathcal{S}_g=0$. An analogue to the proof of Lemma \ref{lem5.2} shows that $\int_{\mathbb{A}_{\kappa}^{d_1}}i_1^*\Psi_{\Sigma}(\mathcal{S}_{h^{\ell}}(\mathcal{S}_g))$ vanishes. One deduces that 
$$\int_{\mathbb{A}_{\kappa}^{d_1}}i_1^*\mathcal{S}_F=\int_{\mathbb{A}_{\kappa}^{d_1}}i_1^*\mathcal{S}_{h^{\ell}}([X_0(g)]).$$
\indent Define a function $g'$ on $\mathbb{A}_{\kappa}^{d_1}\times_{\kappa}\mathbb{A}_{\kappa}^{d_2}$ by setting $g'(x,y)=g(x,y,0)$. Then we have that $g'(0,0)=0$ and $g'(tx,t^{-1}y)=g'(x,y)$ for any $t$ in $\mathbb{G}$. Furthermore, we have an identity in $\mathscr{M}_{X_0(g)}$ as follows
\begin{align*}
[X_0(g)]=\ [X_0(g')]+[\{(x,y,z)\in\mathbb{A}_{\kappa}^{d_1+d_2}\times(\mathbb{A}_{\kappa}^{d_3}\setminus\{0\}) \ |\ g(x,y,z)=0\}].
\end{align*}
As in the proof of Theorem \ref{thm5.1}, since $h^{\ell}$ and $g'$ have no variable in common, we have
$$i_1^*\mathcal{S}_{h^{\ell}}([X_0(g')])=\mathbb{L}^{d_1}\mathcal{S}_{h^{\ell},0}$$ 
in $\mathscr{M}_{\mathbb{A}_{\kappa}^{d_1}\times\mathbb{G}}^{\mathbb{G}}$. It remains to notice that 
$$i_1^*\mathcal{S}_{h^{\ell}}([\{(x,y,z)\in\mathbb{A}_{\kappa}^{d_1+d_2}\times(\mathbb{A}_{\kappa}^{d_3}\setminus\{0\}) \ |\ g(x,y,z)=0\}])=0,$$
because the intersection 
$$i_1(\mathbb{A}_{\kappa}^{d_1})\cap \{(x,y,z)\in\mathbb{A}_{\kappa}^{d_1+d_2}\times(\mathbb{A}_{\kappa}^{d_3}\setminus\{0\}) \ |\ g(x,y,z)=0\}$$
is empty. Thus,
$$\int_{\mathbb{A}_{\kappa}^{d_1}}i_1^*\mathcal{S}_F=\mathbb{L}^{d_1}\mathcal{S}_{h^{\ell},0}$$
in $\mathscr{M}_{\mathbb{G}}^{\mathbb{G}}$. The theorem is proved.
\end{proof}


\end{document}